\documentclass[12pt]{article}
\title{Stationary map coloring
}
\author{Omer Angel\thanks{University of British Columbia; research
    supported by NSERC.}
  \and Itai Benjamini\thanks{Weizmann Institute of Science.}
  \and Ori Gurel-Gurevich\thanks{Microsoft Research.}
  \and Tom Meyerovitch\thanks{Tel Aviv University.}
  \and Ron Peled\thanks{New York University. Partially completed during
    stay at the Institut Henri Poincare - Centre Emile Borel. Research
    supported by NSF Grant OISE 0730136.}}
\date{May 2009}

\usepackage{epsf,epsfig}
\usepackage{amsmath,amsthm,amssymb,amsfonts,latexsym}
\usepackage[margin=3cm]{geometry}

\newtheorem{theorem}{Theorem}[section]

\newtheorem{lem}[theorem]{Lemma}
\newtheorem{lemma}[theorem]{Lemma}
\newtheorem{prop}[theorem]{Proposition}
\newtheorem{cor}[theorem]{Corollary}
\newtheorem{defn}[theorem]{Definition}

\newcommand{\lemref}[1]{Lemma~\ref{L:#1}}
\newcommand{\thmref}[1]{Theorem~\ref{thm:#1}}

\newcommand{\cA}{\mathcal{A}}
\newcommand{\cB}{\mathcal{B}}

\newcommand{\cP}{\mathcal{P}}
\renewcommand{\P}{\mathbb{P}}
\newcommand{\E}{\mathbb{E}}
\newcommand{\N}{\mathbb{N}}
\newcommand{\R}{\mathbb{R}}
\newcommand{\M}{\mathbb{M}}
\newcommand{\Z}{\mathbb{Z}}
\newcommand{\eps}{\varepsilon}
\DeclareMathOperator{\Poi}{Poi}
\DeclareMathOperator{\mex}{mex}

\begin{document}

\maketitle

\begin{abstract}
  We consider a planar Poisson process and its associated Voronoi map. We
  show that there is a proper coloring with $6$ colors of the map which is a
  deterministic isometry-equivariant function of the Poisson process.
  As part of the proof we show that the $6$-core of the corresponding Delaunay
  triangulation is empty.

  Generalizations, extensions and some open questions are discussed.
\end{abstract}

\section{Introduction}

The Poisson-Voronoi map is a natural random planar map. Being planar, a
specific instance can always be colored with 4 colors with adjacent cells
having distinct colors. The question we consider here is whether such a
coloring can be realized in a way that would be isometry-equivariant, that
is, that if we apply an isometry to the underlying Poisson process, the
colored Poisson-Voronoi map is affected in the same way. In other words,
can a Poisson process be equivariantly extended to a colored
Poisson-Voronoi map process? How many colors are needed? Can such an
extension be deterministic?

Extension of spatial processes, particularly of the Poisson process, have
enjoyed a surge of interest in recent years. The general problem is to
construct in the probability space of the given process, a richer process
that (generally) contains the original process. Notable examples include
allocating equal areas to the points of the Poisson process \cite{HP05,
  HHP06,HHP09,Kri,grav1,grav2}; matching points in pairs or other groups
\cite{HP03, DM06, HPPS09, AGH09}; thinning and splitting of a Poisson
process \cite{HLS09, AHS09}. Coloring extensions of i.i.d.\ processes on
$\Z^d$ are considered in \cite{BHSW09}.

We now proceed with formal definitions and statement of the main results.
A non-empty, locally finite subset $S \subset \R^d$ defines a partition of
$\R^d$, called the \emph{Voronoi tessellation}, as follows: The Voronoi
cell $C(x)$ of a point $x\in S$ contains the points of $\R^d$ whose
distance to $S$ is realized at $x$:
\[
C(x) = \{z\in \R^d : d(z,x) = d(z,S)\}.
\]
Points in the intersection $C(x)\cap C(y)$ have equal distance to $x$ and
$y$. It follows that the cells cover $\R^d$ and have disjoint interiors.

For the purposes of coloring, we consider the adjacency graph $G$ of these
cells, with vertices $S$ and edge $(x,y)$ if $C(x)\cap C(y)\neq\emptyset$.
In the case $d=2$, this graph is called the {\em Delaunay triangulation},
and is a triangulation of the plane. (In general, this graph is the
$1$-skeleton of a simplicial cover of $\R^d$.) A $k$-coloring of the
Voronoi tessellation is a proper $k$-coloring of the Delaunay
triangulation, i.e.\ an assignment of one of $k$ colors to each cell so
that adjacent cells have distinct colors. Note that if $S$ does not contain
four or more co-cyclic points, then no more than three cells meet at a
single point. This is a.s.\ the case for the Poisson process. However, for
greater generality one needs the more careful definition, where $(x,y)$ is
an edge if $|C(x)\cap C(y)|>1$. This ensures that the graph is planar.

\begin{figure}[t]
  \centering
  \includegraphics[width=.8\textwidth]{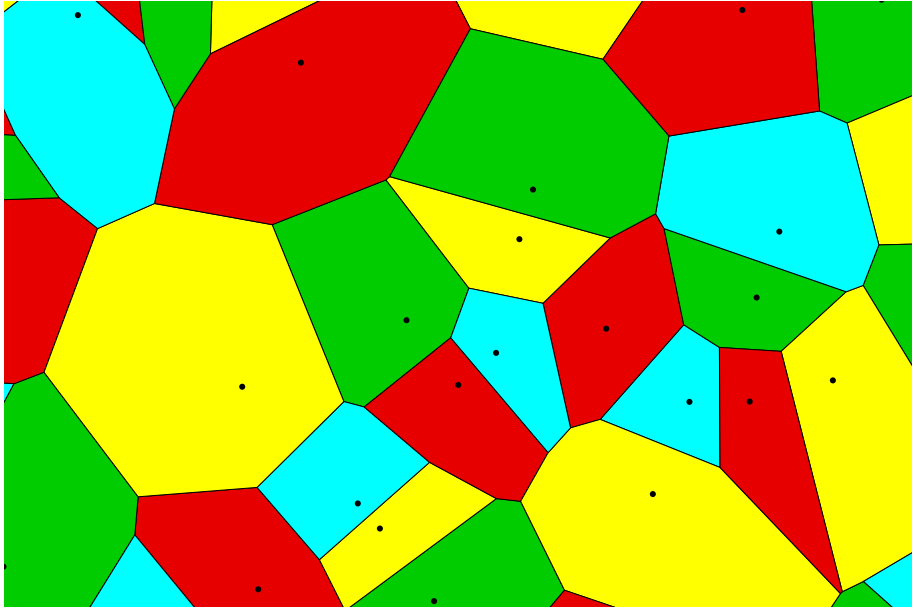}
  \caption{A proper $4$-coloring of a portion of the Poisson-Voronoi map.}
  \label{fig:4-color}
\end{figure}

Given a standard (unit intensity) Poisson process $\cP\subset \R^2$, the
Poisson-Voronoi map is the Voronoi map of its support. By the 4 color
theorem, the Poisson-Voronoi map can always be properly colored with 4
colors. Our main question is whether it is possible to color the
Poisson-Voronoi map in an isometry equivariant way and if so, how many
colors are needed.

To make this precise, let $\M$ be the space of locally finite sets in
$\R^d$, endowed with the local topology and Borel
$\sigma$-algebra.\footnote{It is also common to let $\M$ be the set of
  non-negative integer valued measures on $\R^d$ with
  $\mu(\{x\})\in\{0,1\}$. The distinction will not be important to us.} Let
$\P$ be the probability on $\M$ which is the law of the Poisson process.
Each realization $\cP \in \M$ has the Delaunay graph associated with it. A
(proper) $k$-coloring of $\cP$ is a disjoint partition $\cP = \cup_{i=1}^k
\cP_i$ such that if $x\sim y$ in the Delaunay graph of $\cP$, then $x,y$
are not in the same $\cP_i$. Thus the space of $k$-colored maps is a subset
of $\M^k$.

A {\em deterministic} $k$-coloring scheme of the Voronoi map is a measurable
function $F:\M \to \M^k$ such that $F(\cP)$ is $\P$-a.s.\ a $k$-coloring of
$\cP$. Informally, given the point process, $F$ assigns a color to each
point so that the result is a proper coloring.

A {\em randomized} $k$-coloring scheme of the Voronoi map is a probability
measure $\mu$ on $\M^k$, supported on proper $k$-colorings, such that the law
under $\mu$ of $\cP := \cup_{i=1}^k \cP_i$ is $\P$. Given such a measure
$\mu$, one may consider $\mu$ conditioned on $\cP$. This conditional
distribution is defined $\P$-a.s., and is supported on $k$-colorings of
$\cP$. Thus a randomized $k$-coloring can be interpreted as assigning
to each $\cP\in\M$ a probability measure on $k$ colorings of $\cP$.
Note that any deterministic coloring scheme is also a randomized one, with
$\mu$ being the push-forward of $\P$ by $F$.

A deterministic coloring scheme is said to be {\em isometry equivariant} if every isometry $\gamma$ of $\R^d$, acting naturally on $\M$ and $\M^k$,
has $\gamma F(\cP) = F(\gamma \cP)$. For randomized schemes equivariance is
defined by $\mu \circ \gamma = \mu$. These definitions coincide for
deterministic schemes.

\begin{theorem}\label{thm:main}
  There exists a deterministic isometry equivariant 6-coloring scheme
  of the Poisson-Voronoi diagram in $\R^2$.
\end{theorem}

The requirement of determinism complicates things significantly. In
contrast, we have the much simpler result

\begin{prop}\label{prop:random}
  There exists a randomized isometry equivariant 4-coloring scheme of the
  Poisson-Voronoi diagram in $\R^2$.
\end{prop}

In dimensions other than 2 the problem is not as interesting.

\begin{prop}\label{prop:other_dim}
  In $\R$, there is a randomized isometry equivariant coloring of the
  Poisson-Voronoi map with 2 colors and a deterministic one with 3 colors.
  In both cases this is the best possible.

  In $\R^d$ for $d>2$, the chromatic number of the Poisson-Voronoi map is
  a.s.\ $\infty$.
\end{prop}

The rest of the paper is organized as follows: In section~\ref{sec:outline} we
outline the proof of \thmref{main}, and present our deterministic coloring algorithm and
the two main propositions needed to prove its correctness. In
Section~\ref{sec:gen} we discuss related questions: randomized colorings,
dimensions other than $2$, and mention some open problems.
Section~\ref{sec:proof} contains the proof of our main theorem.

\section{Proof outline} \label{sec:outline}

We outline the proof of Theorem~\ref{thm:main}. The idea is to find an
isometry equivariant adaptation to the Voronoi map of a 6 coloring
algorithm for finite planar graphs, originating in Kempe's attempted proof
of the four color theorem. By Euler's formula it is known that any finite
planar graph $G$ has a vertex of degree at most 5. The algorithm proceeds
by iteratively removing such a vertex until the graph is empty, then
putting back the vertices one by one in reverse order. As each vertex
is put back into the graph, it is assigned a color distinct from those
already assigned to any of its neighbors. Since a vertex has at most 5
neighbors when it is put back, this produces a proper 6 coloring.

To adapt this algorithm to the Poisson-Voronoi isometry equivariant
setting, one must deal with several issues. First, there exist infinitely
many vertices of degree at most 5 and there is no way to pick just one of
them in an isometry-equivariant way. Second, even if we iteratively remove
{\em all} vertices of degree at most 5, the graph will not become empty
after any finite number of steps. Finally, when returning the vertices, it
is not clear in what order to do so (which may be important if some of them
are neighbors). We need a way to order them which is isometry-equivariant.

We overcome these issues by proving that for a Poisson-Voronoi map, the
following two properties hold almost surely. Let $G=(V,E)$ be the Delaunay
graph formed by the Poisson-Voronoi map. For a cell $v\in V$ write $A(v)$
for its area as a planar region. Inductively, define $G_0=G$ and $G_{i+1}$
as the graph formed from $G_i$ by removing all vertices of degree (in
$G_i$) at most $5$.

\begin{prop}\label{removal_prop}
  There exists an integer $M>0$ such that, almost surely, $G_M$ contains
  only finite connected components.
\end{prop}

\begin{prop}\label{area_prop}
  Almost surely, all cells have different areas and there is no infinite
  path in $G$ with decreasing areas.
\end{prop}

We now exhibit a deterministic algorithm which takes as input a graph
$G=(V,E)$ with chromatic number at most $6$ and an area function
$A:V\to\R^+$ satisfying the two propositions above and returns a proper
6-coloring of the graph. Since the algorithm only depends on the graph
structure $G$ and areas $A$ which are preserved by isometries, it is clear
that when applying it to the Delaunay graph of a Poisson-Voronoi map we
will get a deterministic isometry-equivariant 6-coloring.

The algorithm starts with all vertices uncolored. Once a vertex is colored,
its color never changes. Consider first $G_M$. Each of its components is
finite and hence may be colored with 4 colors in an isometry equivariant
way (e.g.\ take the minimal coloring in lexicographic order, when the
vertices of the component are ordered by their area).

Next, having colored $G_k$, we color $G_{k-1}$ inductively. Once $G=G_0$ is
colored, we are done. Consider the vertices of $G_{k-1}\setminus G_k$. Each
has at most 5 neighbors in $G_{k-1}$. We order these vertices by increasing
areas and wish to color them in order, i.e., coloring a vertex $v$ only
after its neighbors of smaller area have been colored. The color of these
neighbors is determined using the same method in an iterative manner.
Proposition~\ref{area_prop} implies that there are just finitely many
vertices that need to be considered before $v$ (see also
Lemma~\ref{L:finite_predecessors}). Hence, going over these finitely many
vertices in order of their areas, we color each one by a color which is
unused by its neighbors (say, the minimal such color) until we finally
color $v$.

Proposition~\ref{removal_prop} is more difficult than
Proposition~\ref{area_prop} and the main lemma required for its proof
(Lemma~\ref{L:finite_radius_6core}) says that if we consider a square of
side length $6R$ and iteratively remove vertices \emph{inside} this square
having degree at most 5, then the square of side length $2R$ with the same
center will eventually become empty with probability tending to $1$ as
$R\to\infty$. This is shown using several probabilistic estimates and uses
of Euler's formula. We then show that for well separated squares of side
length $6R$, the events just described, applied to these squares, are
nearly independent. A small variation on the above event (requiring that
the boxes are also sealed; see below) makes separated boxes completely
independent. Proposition~\ref{removal_prop} then follows by standard
$k$-dependent percolation arguments. Proposition~\ref{area_prop} is proved
using a similar but easier $k$-dependent percolation argument.

As a corollary of the proofs of the above propositions we obtain that our
coloring is finitary with exponential tails. That is, for any given point
$p\in\R^2$, the probability that the color of the cell containing $p$ is not
determined by the points of the Poisson process within a ball of radius $R$ around $p$ is at most
$C e^{-c R}$ for some $C,c>0$.

Note that instead of the area $A$, we could use any other parameter of the
cell (e.g.\ diameter) which satisfies Proposition~\ref{area_prop} (in fact,
one can relax the requirement that all cells have different areas to the
requirement that adjacent cells have different areas). The sole purpose of
$A$ is to induce a well founded order on cells which would ``break ties''
when putting back vertices. We chose to use the
area because it is a very natural parameter to consider, but it is as easy
to prove the required properties for other parameters (see
Section~\ref{area_sec}). A related result is that there is no infinite
path where each Poisson point is the closest to the previous one in the path
\cite{KMN06}.

\section{Generalizations, Extensions and Questions} \label{sec:gen}

In this section we explain some variants and extensions of the
question and settings discussed in our paper.

\subsection{Randomized colorings}

The fact that there is a randomized 4-coloring scheme of the
Poisson-Voronoi map follows from the four color theorem by a soft argument.
This involves an averaging consideration of ergodic theory and works for
any amenable transitive space.

\begin{proof}[Proof of Proposition~\ref{prop:random}]
  The $4$-color theorem implies existence of a measurable function $F$ (not
  necessarily equivariant) which assigns each Voronoi diagram a
  $4$-coloring. E.g. the lexicographically minimal proper coloring is
  easily seen to be a measurable function of the map.

  To get a randomized equivariant coloring, let $\tau_x$ be a
  translation by $x\in\R^2$, $\rho_\theta$ a rotation by $\theta$, and
  $\eps$ the reflection about the $x$ axis. Let $\sigma = \tau_x \circ
  \rho_\theta \circ \eps^u$ be a random isometry, where $u\in\{0,1\}$,
  $\theta\in[0,2\pi]$ and $x\in B(0,R)$ are uniform and independent.

  This defines a probability measure $F_R$ on $4$-colored maps by
  conjugating $F$ by $\sigma$. It is clear (due to compactness of the space
  of distributions over 4-colorings) that $\{F_R\}$ has a subsequential
  weak limit as $R\to\infty$, and any such limit is an isometry equivariant
  4-coloring.
\end{proof}

\paragraph{Explicit Randomized Colorings}

While the previous argument is clearly optimal with respect to the number
of colors used, it is not constructive. It is instructive to consider an
explicit construction with 7 colors. The construction below will be
algorithmic, i.e. there is an algorithm, that determines the color of each
cell by accessing a finite (but unbounded) number of cells along with a
random independent bit for each cell.

As a first stage, we explain how to get an 8-coloring. Start by assigning a
fair coin toss to each cell independently. Consider the subgraph of
$H\subset G$ where an edge is present if its endpoints have the same coin
result. The connected components in this graph are components of site
percolation on $G$ with $p=1/2$. By a result of Zvavitch \cite{Z96}, almost
surely all connected components of both the heads and tails will be finite
(in fact, Bollob{\'a}s and Riordan \cite{BR06} proved that the critical
percolation threshold is indeed $p=\frac{1}{2}$).

Color each ``head'' component independently with colors $\{0,1,2,3\}$ in
some deterministic isometry-equivariant manner which is a function only of the
cells of this component (e.g., again, a lexicographically minimal coloring with vertex order based on cell areas). Color the ``tail'' components with
$\{4,5,6,7\}$. The result is a.s.\ a proper $8$-coloring of $G$. The
randomness comes exclusively from the coin tosses. The color of a cell is
determined by its connected component in $H$ (and the size of the
corresponding cells).

A trick suggested by Gady Kozma \cite{K07} reduces the number of colors
required to $7$ as follows. A finite planar graph embedded in the plane has
a unique unbounded face, called the{\em external} face. Attaching an
additional vertex to the vertices of the external face preserves planarity.
Thus a finite planar graph can be $4$-colored so vertices of the external
face do not use one specified color. Now color the ``heads'' components
using $\{0,1,2,3\}$ so that color $0$ does not appear at vertices of the
external face of any component. Color the ``tails'' components using
$\{0,4,5,6\}$ with the same constraint. Whenever two connected planar
graphs are jointly embedded in the plane, one is contained in the external
face of the other. Thus when a ``tails'' component is adjacent to a
``heads'' component, it is impossible for them to have adjacent vertices
colored $0$, and the coloring is proper.

As noted above, in order to determine the color of any cell, it is
sufficient to know the map structure and the coin-tosses within a ball of a
certain random radius around this cell. In addition, if one modifies the
above algorithm by initially performing fair-independent rolls of a
$3$-sided dice, instead of coin tosses (thus obtaining a proper 10-coloring
in the final outcome, after applying Kozma's trick) then the distribution
of the aforementioned radius will have exponential tails (see \cite{BR06}).
The radius for our deterministic $6$-coloring also has exponential tails, as noted in the proof outline.

\subsection{1-dimensional Poisson-Voronoi map}

The deterministic isometry equivariant chromatic number of a graph may well
be different from its usual chromatic number. For example, consider $Z^d$
translated by a uniform random variable in $[0,1]^d$ and rotated by a uniform random angle in $[0,2\pi]$.
Clearly, the
distribution of this random graph is isometry invariant and it is almost
surely 2-colorable. Yet any deterministic isometry equivariant coloring must assign the same color to all vertices
and hence cannot be proper.

A different example is furnished by the 1-dimensional Poisson-Voronoi
diagram, i.e., the ``Voronoi'' map composed of line segments around the
points of a one-dimensional standard Poisson process. This map is
2-colorable, but we claim that its deterministic isometry equivariant
coloring number is 3. First, it is seen to be at most 3 by considering the
following algorithm: First color green all cells which are shorter than
both their neighbors. Now, from each green cell, proceed to alternately
color its neighbors to the right by red and blue, until the next green cell
is reached. This produces a deterministic translation equivariant proper
3-coloring. To get an isometry equivariant coloring, instead of coloring
red and blue from left to right, start from the shorter of the two green
cells bounding the current stretch of uncolored cells.

The following lemma states that at least 3 colors are needed. A similar
argument appears in Holroyd, Pemantle, Peres and Schramm \cite{HPPS09}.

\begin{lemma} \label{L:1-dim}
  There is no deterministic translation equivariant proper 2-coloring of
  the 1 dimensional Poisson-Voronoi map.
\end{lemma}

\begin{proof}
  In order to reach a contradiction, suppose $\cA$ is such a coloring
  scheme. Since $\cA$ is measurable there exists an integer $L$ and
  another scheme $\cB$, such that the color $\cB$ assigns to the cell at
  the origin depends only on the Poisson process in the interval $[-L,L]$
  and the probability that $\cA$ and $\cB$ assign the same color to a given
  cell is at least $\frac78$. Consider also another point $x>2L$. By
  translation equivariance, the $\cB$-color of the cell of $x$ is
  determined by the Poisson points in $[x-L,x+L]$.

  Hence, with probability at least $\frac34$ the $\cA$-color of both these
  cells is the same as their $\cB$-color. However, The $\cA$-colors of
  these cells determine the parity of the number of cells (i.e.\ points)
  between them. But the parity of the number of points of the Poisson
  process in $[L,x-L]$ is independent of the $\cB$-colors of the origin and
  of $x$, and tends to a uniform on $\{0,1\}$ as $x\to\infty$. Therefore,
  when $x$ is large enough there is a positive probability of a
  contradiction between this parity and the $\cA$-colors of the origin and
  $x$, so this $\cA$ coloring cannot exist.
\end{proof}

We remark that a variant of the $3$-coloring above can be used to color any
invariant point process on $\R$ that is not an arithmetic progression (so
that not all points are isomorphic). Furthermore, the proof of
impossibility with 2 colors also applies to more general processes as we only
use the fact that the parity of the number of points in $[L,x-L]$ is not (nearly)
determined by the process in $[-L,L]$ and $[x-L,x+L]$ for $x$ large enough.

\subsection{Higher dimensional Poisson-Voronoi maps}

A natural generalization of our setting is to consider the 3-dimensional
Poisson-Voronoi diagram. In this case it is not obvious whether one can
properly color the diagram with finitely many colors even without the
isometry equivariant condition. Dewdney and Vranch~\cite{DV77}, and
Preparata~\cite{P77} discovered that $n$ Voronoi cells in $\R^3$ may be all
pairwise adjacent. Indeed, \cite{DV77} shows that in $\R^3$, the Voronoi
cells of $(x_i,x_i^2,x_i^3)_{i=1}^n$ satisfy this for any
$\{x_1,\dots,x_n\}$. Since pairwise adjacency is preserved by sufficiently
small perturbations, and since such configurations a.s.\ appear in the
Poisson process, this implies that the chromatic number of the 3-dimensional
Poisson-Voronoi diagram is almost surely infinite. Higher
dimensional analogues also exist.

Following Proposition~\ref{removal_prop}, one can still ask, as a weaker
result than having an isometry equivariant coloring, what is the minimal
$k$ such that if we iteratively remove all cells having degree at most $k$
we remain with finite components only? Such a $k$ necessarily exists by
arguments similar to those of Proposition~\ref{removal_prop}. (Simulations
indicate that $k=12$ may suffice in $\R^3$.)

\subsection{Ramblings and open questions}

\paragraph{Fewer colors.}
Is there a deterministic $4$-coloring of the Poisson-Voronoi map?
Theorem~\ref{thm:main} shows that $6$ colors suffice, while obviously at
least $4$ are needed. Recent work by Adam Timar \cite{Timar} (in
preparation) shows the existence of deterministic, equivariant 5-colorings
using different methods. Our own methods are close to giving a $5$-coloring
as well, in the following sense: Suppose we define $G_{k+1}$ by removing
from $G_k$ all vertices of degree at most $4$. If
Proposition~\ref{removal_prop} still holds then the same argument gives a
$5$-coloring of $G_0$. To show this, it is enough to prove a statement
similar to Lemma~\ref{L:finite_radius_6core} (roughly put, that the
probability that a large component of the 5-core intersects the boundary of
a box of size $R$ is small enough for some value of $R$). Simulations
suggest that this is indeed the case.

A small difficulty involved in the case of $5$ colors is that not every
vertex is removed at some finite stage. Indeed, the 5-core of the Delaunay
triangulation will not be empty, since it contains finite sub-graphs with
minimal degree 5. The smallest such sub-graph is the dodecahedron,
involving 12 vertices.

Applying the same proof for 4 colors cannot work, since the 4-core of the
Delaunay triangulation has an infinite component. Indeed, a vertex of
degree 3 is necessarily in the interior of the triangle formed by its neighbors. It is straightforward to
check that there are no infinite chains of triangles each one inside the next
(since the probability of long edges decays exponentially; see also
\lemref{rare_squares} below). Therefore, one can consider all the
\emph{maximal} triangles in the Delaunay triangulation. This is also a
triangulation of the plane, since every triangle is contained in a maximal
one, and these are all disjoint. All the vertices of this triangulation
also belong to $G_\infty$ (since none of them are in the interior of
another triangle), and they are all in the same connected component, which
is therefore infinite.

Finally, while we only prove that some $G_M$ (again, deleting vertices of degree $<6$) has only finite connected
components, simulations suggest that $M=2$ suffices while $M=1$ does not.
In fact, it appears sufficient to delete in the second iteration roughly
half the vertices of degree at most 5. Can one prove any of these
assertions?

\paragraph{Other properties of colorings.}
If there is no deterministic $4$-coloring, one could consider intermediate
properties between deterministic and unrestricted randomized colorings. For
example, one may seek colorings that are ergodic, mixing, finitary, etc. Such properties were first brought to our attention by Russ Lyons \cite{L07}.

\paragraph{Other planar processes.}
It might be more interesting to consider other translation or isometry
equivariant graph processes in the plane. These could be the Voronoi
tessellation of some point process or more general planar graph processes.
Except for some obvious counterexamples (see remarks before and after \lemref{1-dim}),
is it true that every such process can be colored deterministically with
$4$ colors? The aforementioned work of Timar \cite{Timar} shows the existence of deterministic 5-colorings.

\paragraph{Hyperbolic geometry.}
What can be done in the hyperbolic plane? Our argument can be adapted to
give a deterministic coloring. However, the number of colors diverges as
the density of the Poisson process tends to $0$, since the average degree
diverges. For high enough density we can get a deterministic $6$-coloring.
Is there a (deterministic or randomized) $k$-coloring with $k$ independent
of the density? While the Poisson-Voronoi map is $4$-colorable by Proposition~\ref{prop:random}, our
randomized constructions use amenability and fail for the hyperbolic plane.

\paragraph{Prescribed color distribution.}
What color distributions are achievable (with deterministic or randomized
colorings)? We only show that coloring schemes exist such that the color of
(say) the cell of 0 is supported on a finite set. If one asks for a
particular distribution the question is interesting also in $\R^d$ for
$d>2$. For example, in $\R^d$, it is possible to get a coloring so that color $i$ appears
with exponentially (in $i$) small probability. What is the
minimal possible entropy of the color of a cell?

\paragraph {Fire percolation.}
Given a set $S_0$ of vertices in the Delaunay triangulation, let $S_k$ be
all vertices at graph distance exactly $k$ from $S_0$. Is it possible to
select a set $S_0$ in a deterministic equivariant manner, so that for all
$k$, $S_k$ has only finite connected components? If the answer is yes, then
coloring the components of $S_k$ for even $k$ with colors $\{0,1,2,3\}$ and
the components for odd $k$ by $\{4,5,6,7\}$ results in a deterministic
8-coloring of the Poisson-Voronoi map. Kozma's aforementioned trick can be
used to get a 7-coloring in this way.

\section{Proof of the main result} \label{sec:proof}

In this section we prove Theorem~\ref{thm:main}. As explained in the proof
outline, the proof is based on Propositions~\ref{removal_prop} and
\ref{area_prop}. These in turn will be proved by reduction to $k$-dependent
percolation. Section \ref{dep_perc_sec} below gives the basic fact about
$k$ dependent percolation we shall need and introduces sealed squares, the
tool which allows us to deduce that events taking place in distant
locations are almost independent. In Section~\ref{area_sec} we prove the
simpler Proposition~\ref{area_prop} and in Section~\ref{removal_sec} the
more difficult Proposition~\ref{removal_prop}. Section~\ref{coloring_sec}
shows how to deduce the main result from the two propositions.

\textbf{Notation:} Throughout we shall denote by $G=(V,E)$ the Delaunay
graph embedded in the plane where $V$ is the set of points of the Poisson
process and the edges are straight lines connecting these points (this can
be seen to be a planar representation of $G$). We will sometimes call the
vertices \emph{centers} and say that a Voronoi cell is \emph{centered} at
its vertex. We also let $A:V \to \R_+$ be the function which assigns to
each vertex the area of the corresponding Voronoi cell. For $x\in\R^2$ we
denote $Q(x,R):=x+[-R,R]^2$, i.e., a square centered at $x$ of side length
$2R$. We let $B_R(x)$ or $B(x,R)$ stand for a closed ball of radius $R$
around $x$ (in the Euclidean metric). We write $d(x,y)$ for the Euclidean
distance between $x,y\in\R^2$. Similarly $d(x,U):=\inf\{d(x,y)\ |\ y\in
U\}$ for sets $U\subseteq\R^2$.

\subsection{Dependent percolation and sealed squares} \label{dep_perc_sec}

A process $\{A_x\}_{x\in\Z^2}$ is said to be {\em $k$-dependent} if for any
sets $S,T \subset \Z^2$ at $\ell^\infty$-distance at least $k$, the
restrictions of $A$ to $S$ and to $T$ are independent. Our processes will
always take values in $\{0,1\}$. Vertices $x\in\Z^2$ with $A_x=1$ are
called {\em open} (and others are {\em closed}). An {\em open component} is
a connected component in $\Z^2$ of open vertices.

A well known result of Liggett, Schonmann and Stacey \cite{LSS96} states
that $k$-dependent percolation with sufficiently small marginals ($\E A_x$)
is dominated by sub-critical Bernoulli percolation. The following simple
lemma is weaker, and is a standard argument in percolation theory. We
include a proof for completeness:

\begin{lemma}\label{lem:k_dependent_percolation}
  For any $k$ there is some $p_0=p_0(k)<1$ such that if
  $\{A_x\}_{x\in\Z^2}$ is $k$-dependent and for all $x$, $P(A_x=1)\leq
  p_0$, then
  \[
  \P(\exists \text{ an infinite open component})=0.
  \]
\end{lemma}

\begin{proof}
  The number of simple paths of length $L$ starting at a given $x\in
  \Z^2$ is bounded by $4^L$. Any simple path of length $L$ contains
  at least $\frac{L}{k^2}$ coordinates which are pairwise $k$-separated.
  Thus, the probability that any given path of length $L$ is open is at
  most $p_0^{L/k^2}$.
  The expected number of open paths originating at $x$ is bounded by
  \[
  4^L \cdot p_0^{L/k^2} = \left(4 p_0^{1/k^2}\right)^L.
  \]
  If $p_0 < 4^{-k^2}$ this quantity tends to $0$ as $L$ tends to infinity.
  However, an infinite open component must contain an open path of any length.
\end{proof}

\begin{defn}
  A set $S\subset\R^2$ is called {\em $\alpha$-sealed} w.r.t.\ the Poisson
  process $V$ if $d(x,V)\le \alpha$ for every point $x\in\partial S$.
\end{defn}

Thus a set is sealed if the point process is not far from any point on the
boundary of $S$. This implies that the Voronoi cells of $V$ which intersect
the boundary of $S$ are centered near the boundary. The purpose of this
notation is that it bounds the dependency between the Voronoi map inside
and outside the set. For a set $S$ we denote
\[
S^\alpha = \{x \in \R^2 : d(x,S)\leq \alpha\}
\]
i.e.\ the closed (Euclidean) $\alpha$-neighborhood of $S$ (so that
$\alpha$-sealed is equivalent to $\partial S\subset V^\alpha$). Note that
being $\alpha$-sealed is determined by $V\cap (\partial S)^\alpha$. We
denote by $S^{-\alpha}$ the points at distance at least $\alpha$ from the
complement $S^c$ (the idea is that if $S=B_R(x)$ then $S^\alpha =
B_{R+\alpha}(x)$ for any $\alpha\ge -R$).

\begin{lemma}\label{L:sealed_independent}
  Condition on the points of $V\cap (\partial S)^\alpha$. On the event that
  $S$ is $\alpha$-sealed, the Voronoi map in $S^{-\alpha}$ is determined by
  the process $V\cap S^\alpha$. Moreover, the cell as well as all neighbors
  of $x\in V\cap S^{-\alpha}$ are contained in $S^\alpha$.
\end{lemma}

\begin{proof}
  The lemma follows from the following simple geometrical fact: If $V\cap
  (\partial S)^\alpha$ is such that $S$ is $\alpha$-sealed, then the center
  of the cell of any $z\in\partial S$ is in $(\partial S)^\alpha$. Thus the
  cells of centers in  $(\partial S)^\alpha$ separate $S^{-\alpha}$ from
  $\R^2 \setminus S^\alpha$.
  It follows that the cell of $x\in S^{-\alpha}$ is contained in $S$, and is
  adjacent only to cells centered in $S^\alpha$.
\end{proof}

Next we argue that squares are likely to be $\alpha$-sealed

\begin{lemma}\label{L:sealed_high_prob}
  The probability that $Q(0,R)$ is not $\alpha$-sealed is at most
  \[
  \lceil 8 R/\alpha \rceil e^{-\pi\alpha^2/4}.
  \]
\end{lemma}

\begin{proof}
  Take an $\alpha/2$ net in $\partial Q(0,R)$, of size $\lceil 8R/\alpha
  \rceil$. Each of these points fails to have a center within distance
  $\alpha/2$ from it with probability $e^{-\pi \alpha^2/4}$. If none fail
  to have such a nearby center then the square is $\alpha$-sealed. A union
  bound gives the claim.
\end{proof}

\subsection{Areas behave --- Proposition~\ref{area_prop}} \label{area_sec}

Our present goal is to prove Proposition~\ref{area_prop}. To this end we
need two properties of the areas of Poisson-Voronoi cells.

\begin{lemma}\label{L:cont_areas}
  Let $\mu$ be the law of the area of the cell containing the origin, then
  $\mu$ is absolutely continuous w.r.t.\ the Lebesgue measure.
\end{lemma}

A partition of $\R_+$ is a finite union $\R_+ = \bigcup_{i<M}
[x_i,x_{i+1})$, given by a sequence $0 = x_0 < x_1 < \cdots < X_M = \infty$.
The following is an immediate corollary of \lemref{cont_areas}.

\begin{cor} \label{cor:partition_areas}
  For any $R,\eps>0$ there is some sufficiently refined partition $\cA$
  of $\R_+$ such that for every interval $I \in \cA$ the probability that
  there exists $v \in V \cap B_R(0)$ with $A(v) \in I$ is at most $\eps$.
\end{cor}

However, just knowing that the area distribution is continuous is not
enough, since the areas of different cells are not independent. For this
reason we also need.

\begin{lemma} \label{L:distinct_areas}
  Almost surely, all cells have different areas.
\end{lemma}

These two lemmas are intuitively obvious, though writing a precise proof is
delicate. It is possible to get a somewhat simpler proof by replacing the
area of a cell by some other quantity. For example, the distance to the
nearest neighbor does not work since some centers have the same distance.
However, total distance to the neighbors in the Delaunay graph does work.

\begin{proof}[Proof of \lemref{cont_areas}]
  The idea of the proof is this: let $x$ be the center of the cell of the
  origin and let $y$ be the center of an adjacent cell. Conditioned on the
  location of all centers other than $y$, and on the direction of the
  vector $y-x$, we get that the area of $x$ is a differentiable function of
  $r=\|y-x\|$, the distance between $x$ and $y$, with positive derivative.
  Thus, $\mu$ conditioned on this $\sigma$-algebra is absolutely continuous
  w.r.t.\ Lebesgue and so $\mu$ itself must also be so.

  To make this precise, we partition $\R^d$ into cubes of size $\eps^d$
  centered around $\eps \Z^d$. We condition on the number of points of the
  Poisson process in each of these cubes. We then use finer and finer
  partitions (say, with $\eps_i=2^{-i}$) until we reach a partition which
  already reveals in what cube lies the center of the cell of the origin
  (i.e.\ $x$) as well as its nearest neighbor (i.e.\ $y$). We then continue
  according to the previous paragraph: we condition on the exact location
  of all points of the Poisson process except $y$ and on the direction of
  $y-x$. After that we get that $A(x)$ is now a monotone function of
  $r=\|y-x\|$ and its derivative is equal to the length of the intersection
  of the cells of $x$ and $y$, which is strictly positive. Since under this
  conditioning, the distribution of $r$ is absolutely continuous w.r.t.\
  Lebesgue measure on some interval we get that the conditioned $\mu$ is
  also absolutely continuous w.r.t.\ Lebesgue and so is $\mu$ itself.
\end{proof}

\begin{proof}[Proof of \lemref{distinct_areas}]
  The proof is similar to that of \lemref{cont_areas}. Fixing any two
  points, $a$ and $b$ we wish to show that the probability that they belong
  to different cells with equal areas is zero. To that end, we find the two
  centers of the cells, $x$ and $y$ and find a third cell, centered at $z$,
  which is adjacent to one of these cells, say, $x$, but not to the other.
  (Such $z$ exists for any $x,y$ in any planar triangulation with no
  unbounded face.) Now $A(x)$ depends on the exact location of $z$, as in
  the proof of \lemref{cont_areas}, but $A(y)$ does not. Of course, all
  this needs to be done using fine partitions, etc.

  The lemma now follows by considering all possible values for $a$ and $b$ with
  rational coordinates.
\end{proof}

Note that the proof of \lemref{distinct_areas} above does not apply as
is to higher dimensions, since in such dimensions, there are configurations with two
distinct cells having the same neighbors. Of course,
\lemref{distinct_areas} itself remain valid.

We now prove Proposition~\ref{area_prop}. The key idea is that cells with
areas in any sufficiently small interval are dominated by sub-critical
percolation.

\begin{proof}[Proof of Proposition~\ref{area_prop}]
  We show that there is some sufficiently refined partition $\cA$ of
  $\R_+$, such that a.s.\ for any $I \in \cA$ there is no infinite path in
  $G$ with all areas in $I$. The proposition will follow since an infinite path
  with decreasing areas will have all areas in the same interval of $\cA$
  from some point on.

  For some $R$ to be determined later, consider the lattice $L=(2R\Z)^2$.
  For an interval $I$, if there is an infinite path of cells with areas in
  $I$, then there is an infinite path $\{x_i\}$ in $L$ so that every
  $Q(x_i,R)$ intersects such a cell. The probability that a square
  intersects a cell with area in $I$ can be made arbitrarily small, but
  these events are not independent. To overcome this we use sealed boxes.

  For an interval $I$, call a point $x\in L$ open if either $Q(x,R)$
  intersects a cell with area in $I$, or if either one of $Q(x,R+\alpha)$
  or $Q(x,R+3\alpha)$ is not $\alpha$-sealed. If there is an infinite path
  in $G$ with areas in $I$ then there is also an infinite open path in $L$.

  The event that the squares are sealed depends only on the Poisson process
  within $Q(x,R+4\alpha)$. We claim that on the event that they are sealed,
  the areas of cells intersecting $Q(x,R)$ also depend only on the process
  in $Q(x,R+4\alpha)$. Taking $\alpha=R/8$ it follows that the process of
  open boxes is $2$-dependent. To see this claim, note that the center of
  any cell intersecting $Q(x,R)$ must be within $Q(x,R+2\alpha)$. The
  second seal implies that the cell of this center is contained in
  $Q(x,R+4\alpha)$ and determined by the process in this box.

  To complete the proof, take some $\eps>0$ so that a $2$-dependent
  percolation with marginal $\eps$ is sub-critical (using
  Lemma~\ref{lem:k_dependent_percolation}). Using
  Lemma~\ref{L:sealed_high_prob}, fix $R$ large enough so that with
  $\alpha=R/8$,
  \[
  \P(Q(x,R+i\alpha) \text{ is not $\alpha$-sealed}) < \eps/3
  \qquad \text{for $i=1,3$.}
  \]
  Next, using Corollary~\ref{cor:partition_areas} take a partition $\cA$
  fine enough that for any $I\in\cA$, the probability that there exists
  $v\in V\cap Q(x,R+2\alpha)$ with area in $I$ is at most $\eps/3$. Then
  for each $I\in\cA$, the probability that any fixed $x$ is open is at most
  $\eps$ and so the process of open points does not contain an infinite
  open path.
\end{proof}

\subsection{Deleting low degree vertices --- Proposition~\ref{removal_prop}}
\label{removal_sec}

In this section we prove Proposition~\ref{removal_prop}. Throughout the
section $R>0$ is a parameter, assumed large enough as needed for the
calculations which follow. We also define the square annuli $A(x,r,R) :=
Q(x,R) \setminus Q(x,r)$.

We now introduce our main object of study in this section:

\begin{defn}
  Inductively, let $G^R_{0}:=G$ and let $G^R_{n+1}$ denote the graph obtained
  from $G^R_{n}$ by deleting all vertices in $Q(0,3R)$ with
  $G^R_{n}$-degree at most $5$. Let $G^R_\infty:=\cap_{n=0}^\infty G^R_n$.
\end{defn}

Thus we iteratively delete vertices of degree at most 5, but only those
vertices contained in a fixed large square. We aim to prove the following

\begin{lemma}\label{L:finite_radius_6core}
  We have $\P\big(G^R_\infty \cap Q(0,R) \neq \emptyset \big)
  \xrightarrow[R\to\infty]{} 0$.
\end{lemma}

\begin{cor}\label{C:finite_radius_6core}
  For any $\eps>0$, there are $R,M$ so that
  $\P\big(G^R_M \cap Q(0,R) \neq \emptyset \big) < \eps$.
\end{cor}

\begin{proof}
  Pick $R$ such that   $\P\big(G^R_\infty \cap Q(0,R) \neq \emptyset \big)
  < \eps$. Since
  \[
  \Big\{G^R_\infty \cap Q(0,R) \neq \emptyset\Big\} = \bigcap_M
  \Big\{G^R_M \cap Q(0,R) \neq \emptyset\Big\},
  \]
  the bound will hold for that $R$
  and sufficiently large $M$.
\end{proof}

Before embarking on the proof of \lemref{finite_radius_6core}, let us
explain how one can get a similar and simpler result when deleting vertices
of degree at most 6 (thus yielding a deterministic $7$-coloring). Suppose
that $G^R_\infty$ contains a vertex in $Q(0,R)$. By \lemref{no_long_edges}
$G^R_\infty$ is unlikely to contain edges longer then $\log R$ within
$Q(0,3R)$. All vertices of $G^R_\infty$ in $Q(0,3R)$ have degree at least
$7$. It is an easy consequence of Euler's formula that a planar graph with
minimal degree $7$ has positive expansion (the boundary of any set is
proportional to its size). This implies (in the absence of long edges) that
the number of vertices of $G^R_\infty$ in $Q(0,3R)$ is exponential in $R$.
Of course, this too is unlikely. When deleting vertices of degree at most
5, the remaining graph has minimal degree 6, which is not as obviously
unlikely. However, this can only happen if $G$ contains a large segment of
a triangular lattice, which we rule out below.

We begin with two combinatorial lemmas on planar maps. For any finite graph
$H$, let $LD=LD(H)$ be the number of vertices of low-degree, namely at most
5. For a finite simple planar map $H$, let $ME=ME(H)$ be the number of
``missing edges'': the number of edges that can be added to the map while
keeping it planar and simple. A face of size $k$ can be triangulated using
$k-3$ edges, after which no further edges can be added, and so $ME = \sum_f
(\deg[f]-3)$ (where the sum also includes the external face and where we
assume $|H|\ge 3$ so that $\deg[f]\ge 3$ for all faces).

\begin{lemma}\label{L:LD_large}
  For any finite, connected and simple planar map $H$ with $|V[H]|\ge 3$
  we have $LD \ge \frac25 ME + \frac{12}{5}$.
\end{lemma}

\begin{proof}
  Add $ME$ edges to make the map into a triangulation. Let $d'_v$ be the
  resulting vertex degrees, than we have $\sum_v (6-d'_v) = 12$ (using
  Euler's formula combined with the triangulation property $3F=2E=\sum
  d'_v$), and therefore $\sum_v (6-d_v) = 12 + 2 ME$. The claim follows
  since low-degree vertices contribute at most 5 to this sum, and high
  degree vertices at most 0, so that $\sum_v (6-d_v) \le 5 L D$.
\end{proof}

\begin{lemma}\label{L:large_6core}
  Fix $\rho>\ell>0$. Let $H$ be a simple planar graph embedded in $\R^2$
  satisfying the following:
  \begin{enumerate}
  \item All vertices in $Q(0,3\rho)$ have degree at least $6$.
  \item All edges of $H$ with an endpoint in $Q(0,3\rho)$ have length at
    most $\ell$.
  \item There exists a vertex of $H$ in $Q(0,\rho)$.
  \end{enumerate}
  Then $H$ has at least $\frac{8\rho^2}{5\ell^2}$ vertices in $Q(0,3\rho)$.
\end{lemma}

Note that the order of magnitude $(\rho/\ell)^2$ is achieved by a
triangular lattice with edge length $\ell$.

\begin{proof}
  We assume that $H$ has only finitely many vertices in $Q(0,3\rho)$ since
  otherwise the conclusion is trivial. Fix a vertex $v\in Q(0,\rho)$. For
  $t\in[\rho,3\rho]$, let $H'_t$ be the
  sub-graph induced by vertices inside $Q(0,t)$, and let $H_t$ be the
  connected component of $v$ in $H'_t$.

  Note that the connected component of $v$ in $H$ is not contained in
  $Q(0,3\rho)$ since otherwise it would be a finite, connected and simple
  planar map with all degrees at least 6 which is impossible by
  Lemma~\ref{L:LD_large}. By our assumptions, all vertices of $H_t$ with
  neighbors in $H\setminus H_t$ (which includes all vertices of degree at
  most 5 in $H_t$) must be in the annulus $A(0,t-\ell,t)$. It follows that
  the external face of $H_t$ surrounds $v$ and exits $Q(0,t-\ell)$ and so
  has degree at least $\frac{2(t-\rho-\ell)}{\ell}$. Thus
  \[
  ME(H_t) \ge \frac{2(t-\rho-\ell)}{\ell} - 3 = \frac{2(t-\rho)}{\ell} - 5.
  \]

  By \lemref{LD_large}, the number of vertices in $A(0,t-\ell,t)$ is at
  least $\frac25 ME(H_t) + \frac{12}{5} \ge \frac45 \frac{t-\rho}{\ell} +
  \frac25$. Let $M = \lfloor 2\rho/\ell\rfloor$. Splitting
  $A(0,\rho,3\rho)$ into annuli $A(0,\rho+(k-1)\ell,\rho+k\ell)$ for
  $k=1,\dots,M$ one finds that the number of vertices of $H$ in
  $Q(0,3\rho)$ is at least
  \[
  1 + \sum_{k=1}^M \left(\frac45 k + \frac25\right) = \frac{2(M+1)^2+3}{5}
  \ge \frac{2(2\rho/\ell)^2}{5}.
  \qedhere
  \]
\end{proof}

Next, a simple lemma showing that long edges in $G$ are unlikely.

\begin{lemma}[No long edges]\label{L:no_long_edges}
  The probability of having an edge of length at least $\ell$ in $E[G]$
  which intersects the square $Q(0,\rho)$ is at most
  \[
  \left( \frac{\sqrt{32}\rho}{\ell}+8 \right)^2 e^{-\ell^2/32}.
  \]
\end{lemma}

\begin{proof}
  Suppose $(x,y)$ were such an edge, then some disc with $x,y$ on its
  boundary has no points in its interior. Consequently at least one of the
  two semi-circles with diameter $(x,y)$ has no points in its interior.
  This implies that there is an empty disc $B_{\ell/4}(z)$ for some $z\in
  Q(0,\rho+\ell)$ ($z$ might be outside $Q(0,\rho)$ since one of $x,y$ may
  be outside the square).

  Cover $Q(0,\rho+\ell)$ by $\left\lceil \frac{\rho+\ell}{\ell/\sqrt{32}}
  \right\rceil^2$ squares of side length $\ell/\sqrt{32}$. It follows that if
  such a long edge exists than one of the squares (the one containing $z$)
  must be empty, and the claim follows.
\end{proof}

We continue by showing that after some low-degree vertices are deleted,
many large holes remain in the graph.

\begin{defn}
  Call a square $Q(x,\rho)$ a \emph{typical square} if there exists some
  vertex $v \in Q(x,\rho)$ such that:
  \begin{enumerate}
  \item $\deg[v] < 6$.
  \item $v$ is not in the interior of any triangle in the Delaunay Graph $G$.
  \end{enumerate}
  Otherwise we call the square
  \emph{rare}.
\end{defn}

To make this clear, the second condition states that there are no
$v_1,v_2,v_3 \in V$ which are pairwise adjacent in $G$ such that $v$ is
contained in the interior of the triangle $(v_1,v_2,v_3)$.

\begin{lem}[Rare squares are rare] \label{L:rare_squares}
  For some $\alpha,\beta>0$ we have $\P(Q(x,\rho) \text{ is rare}) \le
  \alpha\exp(-\beta \rho)$.
\end{lem}

\begin{proof}
  We may assume without loss of generality that $\rho\ge C$ for some large
  $C>0$ (otherwise the claim is trivial). Let $\gamma=\sqrt{\rho}$. The
  square $Q(0,\rho)$ contains at least $c\rho$ disjoint squares
  $Q(x,4\gamma)$ for some $c>0$. Call each of these squares {\em good} if
  it satisfies the following:
  \begin{enumerate}
  \item $Q(x,3\gamma)$ is $\gamma$-sealed,
  \item $Q(x,\gamma)$ contains a vertex $v$ of degree at most 5 which is
    not in the interior of any triangle with vertices in $Q(x,2\gamma)$.
  \end{enumerate}
  Note that by \lemref{sealed_independent}, the event that $Q(x,4\gamma)$
  is good is determined by the Poisson process within it, and so these
  events are all independent. Each square has some probability $p>0$ of
  being good (independent of $\rho$ since long edges are unlikely by Lemma~\ref{L:no_long_edges}),
  so the probability that no
  square within $Q(0,\rho)$ is good is at most $e^{-\beta\rho}$ for some
  $\beta>0$.

  If the low-degree vertex in a good square is contained in a triangle of
  $G$ then an edge of that triangle must have length at least $\gamma$.
  Either the triangle intersects $Q(0,\rho)$, which by
  Lemma~\ref{L:no_long_edges} has probability at most $C_1\rho e^{-\rho/32}
  \le C_2e^{-\rho/33}$ for some $C_1,C_2>0$. Or the triangle contains
  $Q(0,\rho)$ in its interior, in which case for some integer $m\ge 1$, its
  longest edge has length at least $m\rho$ and intersects $Q(0,m\rho)$. By
  a union bound, this has probability at most $\sum_{m=1}^\infty
  C_2e^{-m^2\rho^2/33}\le C_3e^{-\rho^2/33}$ for some $C_3>0$.
\end{proof}

In what follows, define
\begin{align*}
  L &:= \log R, & r &:= R^{1/3}.
\end{align*}
$L$ will be a bound on length of edges that appear (and can be reduced to
$C\sqrt{\log R}$ for large enough $C$). The role of $r$ is more involved,
and there is much freedom in the choice of $r$. Primarily, we consider a
partition of boxes of size of order $R$ into boxes of size $r$. For
simplicity, we assume that $6R/r$ is an odd integer ($R$ can be arbitrarily
large under this condition).

Define for each $x \in r\Z^2$ the square $Q_x:=Q(x,\frac{r}{2})$. Note that
$Q(0,3R)$ is precisely tiled by the boxes $\{Q_x, x\in r\Z^2\cap Q(0,3R)\}$. We
now define several events which we will show to be unlikely.
\begin{align*}
  \Omega_0 &:=\{\text{$G^R_\infty$ has a vertex in $Q(0,R)$}\}, \\
  \Omega_1 &:=\{\text{There exists $e\in E[G]$ of length at least $L$ that
    intersects $Q(0,3R)$}\}, \\
  \Omega_2 &:=\{\text{$Q_x$ is rare for some $x\in r\Z^2\cap Q(0,3R)$}\}, \\
  \Omega_3 &:=\{\text{There exists $x\in r\Z^2 \cap Q(0,3R)$ and $|V\cap
    Q_x| \ge 2r^2$}\}, \\
  \Omega_4 &:=\{|V\cap A| > 2\mbox{Area}(A)\}, \text{ where } A =
  A(0,3R,3R+L).
\end{align*}

Thus \lemref{finite_radius_6core} states that $\P(\Omega_0)$ is small.

\begin{lemma}
  With $L,r$ as above, $\P(\Omega_i) \xrightarrow[R\to\infty]{} 0$ for
  $i=1,2,3,4$.
\end{lemma}

\begin{proof}
  \lemref{no_long_edges} implies that $\P(\Omega_1) = O(R^2 e^{-L^2/32})$
  is small. \lemref{rare_squares} implies $\P(\Omega_2) = O(R^{4/3}
  e^{-\beta r})$ (since there are $(6R/r)^2$ squares to consider).

  $\P(\Omega_3)$ and $\P(\Omega_4)$ are bounded by the fact that
  $\P\big(\Poi(\lambda)>2\lambda\big) \leq e^{-c\lambda}$ for some constant
  $c$. This gives respective bounds $O(R^2 e^{-cr^2})$ and $O(e^{-cRL})$.
\end{proof}

Define the set
\[
S := \left\{x \in r\Z^2 \cap Q\left(0,3R\right) ~:~ Q_x \text{
    is typical and } G^R_{\infty} \cap Q_x \ne \emptyset \right\}.
\]

\begin{lemma}\label{many_boundary_verts_lem}
  There exists $C>0$ such that if $\Omega_1^c$ holds then
  \[
  \big| S \big| \leq C \Big|V\cap A(0,3R,3R+L) \Big|.
  \]
\end{lemma}

\begin{proof}
  Let $H$ be the sub-graph of $G^R_\infty$ induced by vertices in
  $Q(0,3R+L)$. On the event $\Omega_1^c$, the vertices of $H \cap Q(0,3R)$
  all have degree at least 6. Thus low-degree vertices are all in the
  annulus $A(0,3R,3R+L)$ and by Lemma~\ref{L:LD_large},
  \[
  \Big|V\cap A(0,3R,3R+L) \Big| \ge LD(H) > \frac25 ME(H).
  \]

  For each $x\in S$ the square $Q_x$ is typical. Hence there is a vertex
  $v_x\in Q_x$ of degree at most 5 that is not contained in any triangle in
  $G$. The vertex $v_x$ is deleted in the first round and so is not in $H$.
  Let $f_x$ be the face of $H$ surrounding $v_x$. Note that $f_x$ must have
  an edge $e_x$ that intersects $Q_x$, since otherwise $Q_x$ is completely
  in the interior of $f_x$ and there could be no vertex of $H$ in $Q_x$.

  Now, on $\Omega_1^c$, the edge $e_x$ has length at most $L<r$ and
  therefore can intersect at most 3 different squares $Q_x$ (it can
  intersect 3 if it passes near a corner of $Q_x$). Since the face $f_x$
  cannot be a triangle by definition of $v_x$ we deduce that
  \begin{equation*}
    \sum_{\substack{f \text{ face of }H \\ \deg[f]>3}} \deg[f] \ge
    \frac{1}{3}|S|.
  \end{equation*}
  Hence $ME(H) = \sum_{f\in H} (\deg[f]-3) \ge \frac{1}{12}|S|$ proving the
  claim (with $C=30$).
\end{proof}

\begin{proof}[Proof of \lemref{finite_radius_6core}]
  We show that $\Omega_0 \subset \bigcup_{i=1}^4 \Omega_i$. Assume by
  negation that $\Omega_0$ and $\Omega_i^c$ hold for $i=1,2,3,4$. Let $H$
  be the restriction of $G^R_\infty$ to $Q(0,3R+L)$, and apply
  \lemref{large_6core} with $\rho=R$, $\ell=L$. $\Omega_1^c$ and $\Omega_0$
  show that the lemma's hypotheses hold, thus $H$ has at least
  $\frac{8R^2}{5L^2}$ vertices in $Q(0,3R)$.

  On the other hand we show that $H$ is small. Tile $Q(0,3R)$ by boxes
  $Q_x$ with $x\in r\Z^2 \cap Q(0,3R)$. On $\Omega_4^c$,
  Lemma~\ref{many_boundary_verts_lem} implies that $|S| \leq C R L$ for
  some $C$. On $\Omega_3^c$ each of these includes at most $2r^2$ vertices,
  so the number of vertices of $H$ in $Q(0,3R)$ that are in typical boxes
  is at most $2r^2|S| \leq C R L r^2$. On $\Omega_2^c$ there are no
  vertices in rare boxes.

  Thus $\frac{8R^2}{5L^2} \leq C R L r^2$ which is a contradiction
  for $R$ large enough and our choice of $r$ and $L$.
\end{proof}

\begin{proof}[Proof of Proposition~\ref{removal_prop}]
  Define $G^{R,x}_M$ similarly to $G^R_M$, except that low degree vertices
  are deleted in $Q(x,3R)$ instead of $Q(0,3R)$. Consider the following
  dependent percolation process on the lattice $\Lambda=R\Z^2$. A point $x$
  is open in one of 3 cases:
  \begin{enumerate}
  \item The square $Q(x,4R)$ is not $R$-sealed,
  \item $G^{R,x}_M$ has a vertex in $Q(x,R)$,
  \item $G$ has an edge of length at least $R/2$ intersecting $Q(x,R/2)$.
  \end{enumerate}

  We first argue that the event $\{x \text{ is open}\}$ is determined by
  the Poisson process in $Q(x,5R)$, so that the process is 11-dependent.
  Indeed, whether $Q(x,4R)$ is $R$-sealed depends only on the process in
  $Q(x,5R)$. If it is $R$-sealed, the restriction of the Voronoi map to
  $Q(x,3R)$ is determined by the process in $Q(x,5R)$, which determine the
  state of $x$.

  By Lemmas~\ref{L:sealed_high_prob}, \ref{L:finite_radius_6core} and
  \ref{L:no_long_edges}, we can choose $M,R$ so that $\P(x \text{ is
    open})$ is arbitrarily small. In particular, for some $M,R$, using
  Lemma~\ref{lem:k_dependent_percolation}, this percolation is dominated by
  sub-critical percolation, and has no infinite open component.

  Finally, we argue that if there were an infinite component in $G_M$ then
  there would also be an infinite component in our process on $\Lambda$.
  Consider all squares $Q(x,R/2)$ which intersect the edges of some
  infinite open component in $G_M$. For each such $x$, either there is a
  vertex of $G_M$ in $Q(x,R)$, or else the edge that passes through
  $Q(x,R/2)$ has both endpoints outside $Q(x,R)$. Since $G_M\subset
  G^{R,x}_M$, either case implies $x$ is open.
\end{proof}

\subsection{Equivariant Coloring}\label{coloring_sec}

We now use Propositions~\ref{removal_prop} and \ref{area_prop} to construct
a deterministic $6$-coloring scheme. Recall $G_n$ is derived from $G_{n-1}$
by deleting low degree vertices. Define the level of a vertex by
\[
\ell(v) = \max\{n : v\in G_n \}.
\]
Thus a vertex has level 0 iff its degree is at most 5. For neighboring
$v,w$ we direct the edge from $v$ to $w$, and write $v \to w$, if either
$\ell(w)>\ell(v)$ or ($\ell(v)=\ell(w)$ and $A(w) < A(v)$)
(Proposition~\ref{area_prop} gives that no two areas are equal).

Let $\prec$ to be the transitive closure of $\to$. That is, $w\prec v$ iff
there is a finite sequence such that $v=u_0\to u_1 \to \ldots \to u_n=w$.

\begin{lemma}\label{L:finite_predecessors}
  A.s.\ every $v \in V$ has finitely many $\prec$-predecessors (in particular,
  $\prec$ is well founded).
\end{lemma}

\begin{proof}
  We first argue that there is no infinite directed path in $G$. By
  Proposition~\ref{area_prop} there are no infinite $A$-monotone paths, so
  any infinite directed path must have $\ell(v)\to\infty$. However, by
  Proposition~\ref{removal_prop} there are no infinite paths with $\ell>M$.

  Our conclusion then follows from K\"{o}nig's lemma: A locally finite tree
  with no infinite paths is finite.
\end{proof}

From this we get:

\begin{prop}
  There exists a unique function $f:V \to \{0,\dots,5\}$ determined by the
  recursive formula:
  \[
  f(u) = \mex \{ f(v) : u\to v \}
  \]
  where $\mex S = \min(\N \setminus S)$ is the minimal excluded integer
  function.
\end{prop}

\begin{proof}
  The proof is by induction on $\prec$, which is a well founded order by
  \lemref{finite_predecessors} (see e.g.\ \cite[Chapter 3]{Kunen}). Any $u
  \in V$ has at most 5 neighbors $v$ with $\ell(v) \ge \ell(u)$, so $|\{ v
  : u\to v\}| \le 5$ and so $f(u) < 6$ is well defined.

  Uniqueness holds since $f(u)$ is determined by $\{f(v) : v\prec u\}$.
\end{proof}

\thmref{main} now follows, since $\cP_i = f^{-1}(i)$ defines a deterministic,
isometry equivariant $6$-coloring. Note that the resulting coloring is
finitary, that is, for every $x \in \R^2$ there exists a finite (but
random) $R>0$ such that the color of the cell containing $x$ is a function
of the Poisson process restricted to $B_R(x)$. Indeed, to determine $f(v)$
for $v \in V$, it is sufficient to know the graph $G$ induced on the
$\prec$-predecessors of $v$. Furthermore, there exist $C,c>0$ such that
$\P(R>s) \leq C e^{-c s}$. This is the case because
Propositions~\ref{removal_prop} and \ref{area_prop} are proved using
domination by sub-critical percolation.

\bibliography{colouring}

\end{document}